\documentclass[12pt]{amsart} 
\usepackage{amsmath,amsthm,amssymb,graphicx,marvosym,lipsum}
\usepackage{lipsum}
  
\newtheorem{theorem}{Theorem}[section]
\newtheorem{Lemma}[theorem]{Lemma}

\newtheorem{Corollary}[theorem]{Corollary}
\newtheorem{Definition}[theorem]{Definition}
\newtheorem{Example}[theorem]{Example}
\newtheorem{Remark}[theorem]{Remark}

\newcommand\blfootnote[1]{%
  \begingroup
  \renewcommand\thefootnote{}\footnote{#1}%
  \addtocounter{footnote}{-1}%
  \endgroup
}
\begin{document}
\title{Topological Stability in Paired Dynamical Systems}   
\author[]{Abdul Gaffar Khan$^{2,*}$, Pramod Kumar Das$^{3}$, Tarun Das$^{1}$}                 

\begin{abstract}
We study the classical topological dynamical notions of shadowing and topological stability from a viewpoint of paired dynamical system $(f,g)$, where $f$ and $g$ are uniform equivalences on a metric space $X$. We observe that if $g$ is equicontinuous and commutes with $f$, then the study of shadowing for $(f,g)$ is reduced to the study of the classical shadowing for $g^{-1}f$. The fact that these assumptions are sufficient is justified through examples. Finally, we prove that if $f$ is an expansive homeomorphism on a relatively compact metric space, then any pair $(f,g)$ with shadowing is topologically stable.     
\end{abstract}
\maketitle
\begin{flushleft}
\end{flushleft}

\blfootnote{\hspace*{-0.5cm} Abdul Gaffar Khan, gaffarkhan18@gmail.com \\
Pramod Kumar Das, pramodkumar.das@nmims.edu\\
Tarun Das, tarukd@gmail.com\\
\textit{$^{*}$Corresponding author}\\ 
\textit{$^{1}$Department of Mathematics, Faculty of Mathematical Sciences, University of Delhi, Delhi, India.}\\ 
\hspace*{0.11cm}\textit{$^{2}$Kirori Mal College, Department of Mathematics, University of Delhi, Delhi-110007.}\\
\hspace*{0.11cm}\textit{$^{3}$School of Mathematics, Applied Statistics and Analytics,  Narsee Monjee Institute of Management Studies, Vile Parle, Mumbai-400056,  India.}  
}
\section{Motivation and Preliminaries}
In the middle of the twentieth century, Utz defined a self-homeomorphism $f$ on a metric space $(X, d)$ to be unstable if there exists a positive constant $\mathfrak{c}$ such that for every pair of distinct points $x,y\in X$, there exists an $n\in\mathbb{Z}$ satisfying $d(f^n(x),f^n(y))>\mathfrak{c}$ \cite{UU}. Such homeomorphisms are known as expansive homeomorphisms in the current literature. The expansivity of symbolic flows has been recognized several years before the above formal introduction. Since any two distinct orbits cannot remain close to each other for all time in any expansive system, it can be seen as a stronger form of sensitivity, an important component of chaos. 
\par
\vspace*{0.1cm}

Recall that a self-homeomorphism $f$ on a metric space $(X, d)$ is said to have shadowing property, if for every $\epsilon>0$, there exists a $\delta>0$ such that for every sequence $\lbrace x_n\rbrace_{n\in\mathbb{Z}}$ satisfying $d(f(x_n),x_{n+1})\leq \delta$, for each $n\in\mathbb{Z}$, there exists an $x\in X$ satisfying $d(f^n(x),x_n)\leq \epsilon$, for each $n\in\mathbb{Z}$. In modern days, one uses computer to study extremely complex dynamical systems. This computer simulation provides us with approximate orbit instead of actual orbit. The presence of shadowing guarantees that the behaviour of an actual orbit must reflect into the nearby approximate orbit. One can refer \cite{AHT} for detailed study of expansivity and shadowing property.    
\par
\vspace*{0.1cm}

Throughout this paper $(X,d)$, $(Y, p)$ denotes metric spaces and $Id$ denotes the identity map on $X$. We assume that every map $f:X\rightarrow X$ is a uniform equivalence i.e., both $f$ and $f^{-1}$ are uniformly continuous. We say that $(X,(f,g))$ is a paired dynamical system or simply $(f,g)$ is a paired system on $X$ whenever $f$ and $g$ are uniform equivalences on $X$. The closed balls and open balls centred at $x$ of radius $\delta$ are denoted by $B[x, \delta]$ and $B(x, \delta)$ respectively. A metric space $X$ is said to be relatively compact if every bounded subset of $X$ is contained in a compact set which is equivalent to saying that every closed ball of finite radius in $X$ is compact. 

\begin{Definition}
Let $(f,g)$ be a paired system on $X$. Then a sequence $\lbrace x_{n} \rbrace_{n\in \mathbb{Z}}$ of elements of $X$ is said to be a $\delta$-pseudo orbit of $(f,g)$ if $d(f(x_{n}), g(x_{n+1}))$ $\leq \delta$, for each $n\in \mathbb{Z}$. Further, a $\delta$-pseudo orbit is said to be $\epsilon$-traced if there exists an $x\in X$ such that $d(f^{n}(x),g^{n}(x_n))\leq \epsilon$, for each $n\in \mathbb{Z}$. We say that $(f,g)$ has paired shadowing property (PSP) if for each $\epsilon >0$, there exists a $\delta >0$ such that each $\delta$-pseudo orbit of $(f, g)$ can be $\epsilon$-traced by some point of $X$. 
\label{D1}
\end{Definition}

The observation discussed below has motivated us to introduce paired shadowing property. This discussion also expresses the connections between paired shadowing and classical shadowing. 
\par
\vspace*{0.1cm}

For $A\subset X\times X$, set $A_{P_{1}} = \lbrace x \mid (x, y)\in A\rbrace$ i.e., the image of the projection map from $A$ to the first coordinate of $A$ and set $A_{P_{2}} = \lbrace y \mid (x, y)\in A\rbrace$ i.e., the image of the projection map from $A$ to the second coordinate of $A$. 
For a homeomorphism $h$ on $X$, $G(h) = \lbrace (x, h(x)) \mid x\in X \rbrace$ denotes the graph of $h$. 
Further, assume that $x\in X$, $\delta > 0$ and $n\in \mathbb{Z}$. Denote the following
\begin{center}
\hspace{0.4cm}$B[h(x), \delta] = x^{+}(\delta) = \lbrace z\in (G(Id))_{P_{1}} \mid  Id(z) \in (G(Id))_{P2} \cap B[h(x), \delta] \rbrace, $ \\ 
$h^{-1}(B[x,\delta]) = x^{-}(\delta) = \lbrace z\in (G(h))_{P_{1}} \mid h(z)\in (G(h))_{P2} \cap B[Id(x), \delta] \rbrace$, \\ 
$h^{-n}(B[x,\delta]) = h_{x}^{n}(\delta) = \lbrace z\in (G(h^{n}))_{P_{1}} \mid  h^{n}(z) \in (G(h^{n}))_{P2} \cap B[(Id)^{n}(x), \delta] \rbrace.$
\end{center}
Note that the notion of shadowing of a homeomorphism $f$ can be stated as for every $\epsilon>0$, there exists a $\delta>0$ such that whenever $\lbrace x_n\rbrace_{n\in\mathbb{Z}}$ satisfies $x_{n}\in x_{n-1}^{+}(\delta)\cap x_{n+1}^{-}(\delta)$, for each $n\in \mathbb{Z}$, we have $\bigcap\limits_{n\in\mathbb{Z}}f_{x_{n}}^{n}(\epsilon)\neq \phi$. This observation directs the dependence of shadowing on the graph of identity map and its self composition which is a natural motivation to study the variant of this dynamical notion by replacing the graph of identity map by a uniform equivalence $g$, which turns out to be equivalent to Definition \ref{D1}. In  words, PSP requires that the $n^{th}$ time orbit under $g$ of the approximated $n^{th}$ time piece can be seen closely to $n^{th}$ time position of a true orbit under $f$. 
\par
\vspace*{0.1cm}

Since the composition of any homeomorphism on the real line having shadowing with its inverse does not have the shadowing, we can conclude that the composition of two maps with shadowing need not have shadowing. The study of shadowing for paired systems indicates where such possibilities could occur. 
\par
\vspace*{0.1cm}

Recall that $f\times g$ has shadowing if and only if $f$ and $g$ have shadowing property. But we can not say that the paired system $(f, g)$ has shadowing if and only if $f$ and $g$ has shadowing (please see Example \ref{E1}). Thus the study of paired system $(f, g)$ and the cartesian system $f\times g$ can be used to obtain distinct information about the effect of dynamical behaviour of one system on another.
\par
\vspace*{0.1cm}

Motivated from the above observations, we introduce paired dynamical systems and explore the study of the relative behaviour of uniform equivalences via various dynamical notions in greater details.
\par
\vspace*{0.1cm}

We say that a paired system $(f, g)$ on $X$ is commutative if $f\circ g = g\circ f$. Note that $(f, g)$ is commutative if and only if $(f, g^{-1})$ is commutative  if and only if $(f^{-1}, g)$ is commutative if and only if $(f^{-1}, g^{-1})$ is commutative. 
\par
\vspace*{0.1cm}

Recall that a homeomorphism $f$ on $X$ is equicontinuous if for each $\epsilon > 0$, there exists a $\delta > 0$ such that if $x,y\in X$ satisfy $d(x, y) < \delta$, then $d(f^{n}(x), f^{n}(y)) < \epsilon$, for each $n\in \mathbb{Z}$.

\section{Paired Shadowing Property}
In this section, we discuss the relation between paired shadowing property of paired system $(X, (f, g))$ and the classical shadowing of autonomous system $(X, g^{-1}f)$. Then several properties of paired system with paired shadowing is proved. Firstly, we need the following definition.

\begin{Definition}
We say that paired systems $(X,(f_{1},g_{1}))$ and $(Y,(f_{2},g_{2}))$ are uniformly conjugate if there exists a uniform equivalence $\Phi : X\rightarrow Y$ such that $\Phi\circ f_{1} = f_{2}\circ \Phi$ and $\Phi\circ g_{1} = g_{2}\circ \Phi$. We say that a property of a paired system is a uniform dynamical property if it is preserved under uniform conjugacy i.e. if $(X, (f,g))$ has property $P$, then every uniformly conjugated system to $(X,(f,g))$ has property $P$.  
\label{D3}
\end{Definition}

\begin{theorem}
Let $(f,g)$ be a commutative paired system on $X$ and $g$ be equicontinuous. Then $(f, g)$ has paired shadowing if and only if $g^{-1}f$ has shadowing.
\label{T1}
\end{theorem}
\begin{proof}
Assume that $(f, g)$ has paired shadowing. For a given $\epsilon > 0$, choose $0 < \eta < \epsilon$ such that if $x,y\in X$ satisfy $d(x, y)< \eta$, then $d(g^{n}(x), g^{n}(y)) < \epsilon$, for each $n\in \mathbb{Z}$. Choose a $\gamma > 0$ such that each $\gamma$-pseudo orbit of $(f, g)$ can be $\eta$-traced through $(f, g)$. Choose a $\delta > 0$ such that if $x,y\in X$ satisfy $d(x, y)< \delta$, then $d(g(x), g(y)) < \gamma$. One can observe that each $\delta$-pseudo orbit $\lbrace x_{n}\rbrace_{n\in \mathbb{Z}}$ of $g^{-1}f$ is a $\gamma$-pseudo orbit of $(f, g)$. Since $(f, g)$ has  paired shadowing, we get that every $\eta$-tracing point of $\lbrace x_{n}\rbrace_{n\in \mathbb{Z}}$ through $(f, g)$ is an $\epsilon$-tracing point of $\lbrace x_{n}\rbrace_{n\in \mathbb{Z}}$ through $g^{-1}f$ implying that $g^{-1}f$ has shadowing. Similarly, we can prove that the converse holds.
\end{proof}

\begin{Corollary}
Let $(f,g)$ be a commutative paired system on a compact metric space $X$. If $g$ is equicontinuous, then $(f, g)$ has paired shadowing if and only if $(f^{k}, g^{k})$ has paired shadowing, for some $k\in \mathbb{Z}\setminus \lbrace 0\rbrace$.
\label{C1}
\end{Corollary}
\begin{proof}
Recall that a homeomorphism $h:X\rightarrow X$ has shadowing if and only if $h^{k}$ has shadowing for some $k\in \mathbb{Z}\setminus \lbrace 0\rbrace$. Now proof follows from Theorem \ref{T1}.
\end{proof}

Following examples justify that if we drop either the commutativity of paired system $(f, g)$ or the equicontinuity of $g$, then the Theorem \ref{T1} does not hold true.

\begin{Example}
Let $X = \mathbb{R}$ be the real line equipped with the Euclidean metric. For $\mathsf{a} > 0$ and $\mathsf{a}\neq 1$, define $f_{\mathsf{a}} : X\rightarrow X$ by $f_{\mathsf{a}}(x) = \mathsf{a}x$, for each $x\in X$. Clearly, $f_{\mathsf{a}}^{-1} = f_{\frac{1}{\mathsf{a}}}$.  We claim that $(f_{2}, f_{\frac{1}{2}})$ does not have PSP. 
Choose $\epsilon = \frac{1}{2}$. For each $k\in \mathbb{N}$. set $\delta_{k} = \frac{1}{k}$ and define $\lbrace x_{n}^{k}\rbrace_{n\in \mathbb{Z}}$ be  such that 
\begin{center}
$x_{n}^{k} =  \begin{cases} 
		0 & \text{if\space} n\leq 0 \\
		\frac{2}{5k}(1-\frac{1}{6^{n}}) &  \text{if\space} n\geq 1
		\end{cases}.$
\end{center}
Clearly $\lbrace x_{n}^{k}\rbrace_{n\in \mathbb{Z}}$ forms a $\delta_{k}$-pseudo orbit of $(f_{2}, f_{\frac{1}{2}})$. Note that if $x^{k}$ is an $\epsilon$-tracing point of $\lbrace x_{n}^{k}\rbrace_{n\in \mathbb{Z}}$, then $d(f_{2}^{n}(x^{k}), f^{n}_{\frac{1}{2}}(x_{n}^{k})) \leq \frac{1}{2}$, for each $n\in \mathbb{Z}$ implying that $\frac{-3}{2} \leq 2^{n}(x^{k}) \leq \frac{3}{2}$, for each $n\in \mathbb{N}$ which is not possible as $2^{n}x^{k}\rightarrow \infty$ as $n\rightarrow \infty$. Hence $(f_{2}, f_{\frac{1}{2}})$ does not have PSP.
Since $f_{2}f_{\frac{1}{2}}^{-1}= f^{4}$ has shadowing property, we conclude that if $(f, g)$ is commutative but $g$ is not equicontinuous, then shadowing of $g^{-1}f$ does not imply that $(f, g)$ has PSP. 
\label{E1} 
\end{Example}

\begin{Remark}
We can use similar arguments as given in Example \ref{E1} to show that for each $\mathsf{a} > 0$ and $\mathsf{a}\neq 1$, paired systems $(f_{\mathsf{a}}, f_{\frac{1}{\mathsf{a}}})$ and $(Id, f_{\mathsf{a}})$ on $X$ does not have PSP. 
\end{Remark}

\begin{Example}
Let $X = \prod_{-\infty}^{\infty} X_{i}$ where $X_{i} = \lbrace 0, 1\rbrace$, for each $i\in \mathbb{Z}$ and be equipped with the metric $d(x, y) = \sum_{i = -\infty}^{\infty} \frac{|x_{i} - y_{i}|}{2^{|i|}}$, where $x = (x_{i}), y = (y_{i})\in X$. 
Let $f$ be the shift map defined as $f(x) = y$, where $y_{i} = x_{i+1}$, for each $i\in \mathbb{Z}$ and for each $x = (x_{i})\in X$. We claim that $(Id, f)$ does not have paired shadowing. 
Choose $0< \epsilon < \frac{1}{4}$ and set $\delta_{k} = \frac{1}{k}$, for each $k\in \mathbb{N}$. Note that $d(x, y) < \frac{1}{4}$ only when $x_{i} = y_{i}$, for all $-2\leq i\leq 2$. Choose the smallest $p_{k}\geq 0$ such that $d(x, y) < \delta_{k}$ whenever $x_{i} = y_{i}$, for all $-p_{k}\leq i\leq p_{k}$. 
For each $k\in \mathbb{N}$, define a sequence $\lbrace x_{n}^{k}\rbrace_{n\in \mathbb{Z}}$ such that $x_{n}^{k} = (0)$, for all $n\leq 0$, $(x_{n-1}^{k})_{i} = (x_{n}^{k})_{i+1}$, for all $-p_{k}\leq i\leq p_{k}$ and for all $n\geq 1$ and $(x_{p_{k}}^{k})_{p_{k}+1} \neq (x_{p_{k}+1}^{k})_{p_{k}+2}$. Clearly, $\lbrace x_{n}^{k}\rbrace_{n\in \mathbb{Z}}$ forms a $\delta_{k}$-pseudo orbit of $(Id, f)$. 
Note that if $x^{k}\in X$ is an $\epsilon$-tracing point of $\lbrace x_{n}^{k}\rbrace_{n\in \mathbb{Z}}$, then $d(x^{k}, f^{n}(x_{n}^{k})) < \frac{1}{4}$, for each $n\in \mathbb{Z}$. 
Therefore $(x^{k})_{i} = (x_{n}^{k})_{n+i}$, for all $-2\leq i \leq 2$ and for each $n\in \mathbb{Z}$ implying that $(x_{p_{k}}^{k})_{p_{k}+1} = (x^{k})_{1} = (x_{p_{k}+1}^{k})_{p_{k}+2}$ which is a contradiction. Hence $(f, g)$ does not have PSP. Since $f^{-1}$ has shadowing property, we get that if $(f, g)$ is commutative but $g$ is not equicontinuous then shadowing of $g^{-1}f$ does not imply that $(f, g)$ has shadowing even on a compact space. 
\label{E2}
\end{Example}

\begin{Remark}
If $f$ denotes the shift map defined in Example \ref{E2}, then we can use similar arguments as given in Example \ref{E2} to show that the paired system $(f, f)$ on $X$ does not have PSP. 
\end{Remark}

\begin{Example}
Let $X = \lbrace -1, 0, 1\rbrace$ be equipped with the discrete metric. Define $f: X\rightarrow X$ by $f(0) = 1$, $f(1) = -1$ and $f(-1) = 0$. Define $g: X\rightarrow X$ by $g(0) = 1$, $g(1) = 0$ and $g(-1) = -1$. Clearly $(f, g)$ is not commutative and $g^{-1}f$ has shadowing. It is easy to check that for any $0 < \epsilon < 1$ and $0 < \delta < 1$, the sequence $\lbrace x_{n}\rbrace_{n\in \mathbb{Z}}$ defined as $x_{2n} = -1$ and $x_{2n+1} = 1$, for each $n\in \mathbb{Z}$ forms a $\delta$-pseudo orbit of $(f, g)$ but it can not be $\epsilon$-traced through $(f, g)$ by any point of $X$. Hence if $g$ is equicontinuous but $(f, g)$ is not commutative, then the shadowing of $g^{-1}f$ does not imply that $(f, g)$ has paired shadowing.
\label{E3} 
\end{Example}
\vspace*{0.2cm}

\begin{theorem}
Let $(f, g)$ and $(f', g')$ be paired systems on $(X, d)$ and $(Y, p)$ respectively. Then, the following statements are true:
\begin{enumerate}
\item[(i)] Assume that $X \times Y$ is equipped with the maximum metric $q$. Then $(F, G) = (f\times f', g\times g')$ has PSP if and only if $(f, g)$ and $(f', g')$ have PSP.
\item[(ii)] If $(f,g)$ is commutative, then $(f, g)$ has PSP if and only if $(f^{-1}, g^{-1})$ has PSP.
\item[(iii)] If $(f,g)$ is commutative, then $(f, g)$ has PSP if and only if $(f^{k}, g^{k})$ has PSP, for all $k\in \mathbb{Z}\setminus \lbrace 0\rbrace$.
\end{enumerate}
\label{T2}
\end{theorem}
\begin{proof}
Let $(f, g)$ and $(f', g')$ be paired systems on $(X, d)$ and $(Y, p)$. 
\begin{enumerate}
\item[(i)] Assume that $X \times Y$ is equipped with the maximum metric $q$ defined by $q((x_{1}, y_{1}), (x_{2},y_{2})) = \max \lbrace d(x_{1}, x_{2}), p(y_{1}, y_{2})\rbrace$, for all $(x_{1}, y_{1}), (x_{2}, y_{2})\in X\times Y$. For a given $\epsilon > 0$, choose a $\delta > 0$ by PSP of $(F, G)$. Let $\lbrace x_{n} \rbrace_{n\in \mathbb{Z}}$ and $\lbrace y_{n}\rbrace_{n\in \mathbb{Z}}$ be $\delta$-pseudo orbits of $(f, g)$ and $(f', g')$ respectively. Clearly, $\lbrace (x_{n}, y_{n}) \rbrace_{n\in \mathbb{Z}}$ forms a $\delta$-pseudo orbit of $(F, G)$ and hence can be $\epsilon$-traced through $(F, G)$ by some point $(x, y)\in X \times Y$. 
Note that $x$ is an $\epsilon$-tracing point of $\lbrace x_{n} \rbrace_{n\in \mathbb{Z}}$ through $(f, g)$ and $y$ is an $\epsilon$-tracing point of $\lbrace y_{n} \rbrace_{n\in \mathbb{Z}}$ through $(f', g')$. Therefore $(f, g)$ and $(f', g')$ have PSP. 
Conversely, for a given $\epsilon > 0$, choose a $\delta > 0$ by PSP of $(f, g)$ and $(f', g')$. 
We observe that if $\lbrace (x_{n}, y_{n}) \rbrace_{n\in \mathbb{Z}}$ is a $\delta$-pseudo orbit of $(F, G)$, then $\lbrace x_{n}\rbrace_{n\in \mathbb{Z}}$ is a $\delta$-pseudo orbit of $(f, g)$ and $\lbrace y_{n}\rbrace_{n\in \mathbb{Z}}$ is a $\delta$-pseudo orbit of $(f', g')$. 
Choose an $x\in X$ and $y\in Y$ such that $\lbrace x_{n} \rbrace_{n\in \mathbb{Z}}$ and $\lbrace y_{n} \rbrace_{n\in \mathbb{Z}}$ can be $\epsilon$-traced by $x$ and $y$ through $(f, g)$ and $(f', g')$ respectively. Note that $\lbrace (x_{n}, y_{n}) \rbrace_{n\in \mathbb{Z}}$ can be $\epsilon$-traced by $(x, y)$ through $F\times G$. 
Hence $(F, G)$ has PSP.
 
\item[(ii)] Since it is enough to prove the one way, we assume that $(f, g)$ is a commutative paired system with PSP. 
For a given $\epsilon > 0$, choose a $\delta > 0$ such that each $\delta$-pseudo orbit of $(f, g)$ can be $\epsilon$-traced by some point of $X$. 
Choose $0 < \delta' < \delta$ such that if $x,y\in X$ satisfy $d(x, y) < \delta'$, then $d(g(x), g(y)) < \delta$. 
Choose $0 < \delta'' < \delta'$ such that if $x, y\in X$ satisfy $d(x, y) < \delta''$, then $d(f(x), f(y)) < \delta'$. 
Note that if $\lbrace x_{n} \rbrace_{n\in \mathbb{Z}}$ is a $\delta''$-pseudo orbit of $(f^{-1}, g^{-1})$, then $\lbrace x'_{n} = x_{-n} \rbrace_{n\in \mathbb{Z}}$ is a $\delta$-pseudo orbit of $(f, g)$. Also that every $\epsilon$-tracing point of $\lbrace x'_{n} \rbrace_{n\in \mathbb{Z}}$ through $(f, g)$ is an $\epsilon$-tracing point of $\lbrace  x_{n} \rbrace_{n\in \mathbb{Z}}$ through $(f^{-1}, g^{-1})$.
Hence $(f^{-1}, g^{-1})$ has PSP. 

\item[(iii)] Since converse follows trivially, we prove the forward implication only. Suppose that $(f, g)$ has PSP and fix a $k \geq 1$. For a given $\epsilon > 0$, choose $0 < \eta < \epsilon$ such that each $\eta$-pseudo orbit of $(f, g)$ can be $\epsilon$-traced through $(f, g)$ by some point of $X$. 
Choose $0 < \delta < \eta$ such that if $x, y\in X$ satisfy $d(x, y)< \delta$, then $d(g^{-k+1}(x), g^{-k+1}(y)) < \eta$. 
Consider a $\delta$-pseudo orbit $\lbrace x_{n}\rbrace_{n\in \mathbb{Z}}$ of $(f^{k}, g^{k})$ i.e. $d(f^{k}(x_{n}), g^{k}(x_{n+1})) < \delta$, for each $n\in \mathbb{Z}$. Clearly, $d(g^{-k+1}f^{k}(x_{n}), g(x_{n+1})) < \eta$, for each $n\in \mathbb{Z}$. 
Define $y_{kn + j} = g^{-j}f^{j}(x_{n})$, for each $n\in \mathbb{Z}$ and for all $0 \leq j \leq (k-1)$. 
Note that for each $n\in \mathbb{Z}$, given $0 \leq j \leq (k-2)$ and for $i = kn+j$, we have $d(f(y_{kn+j}), g(y_{kn+j+1})) = d(g^{-j}f^{j+1}	(x_{n}), gg^{-j-1}f^{j+1}(x_{n})) = 0 < \eta$. 
For $i = kn+k-1$, we have $d(f(y_{i}), g(y_{i+1})) = d(g^{-k+1}f^{k}(x_{n}), g(x_{n+1})) < \eta$. 
Therefore $\lbrace y_{n} \rbrace_{n\in \mathbb{Z}}$ is an $\eta$-pseudo orbit of $(f,g)$. 
Choose a $y\in X$ such that $d(f^{n}(y), g^{n}(y_{n})) < \epsilon$, for each $n\in \mathbb{Z}$. 
Since $y_{kn} = x_{n}$, for each $n\in \mathbb{Z}$, we get that $d(f^{kn}(y), g^{kn}(x_{n}))< \epsilon$, for each $n\in \mathbb{Z}$. 
Hence $(f^{k}, g^{k})$ has PSP. Since $k$ is chosen arbitrary, we use (ii) to conclude that $(f^{k}, g^{k})$ has PSP, for all $k\in \mathbb{Z}\setminus \lbrace 0\rbrace$.
\end{enumerate}
\end{proof}

\begin{theorem} \label{T3}
For paired dynamical systems, PSP is a uniform dynamical property. 
\end{theorem}
\begin{proof}
Let $(X,(f,g))$ and $(Y,(f',g'))$ be uniformly conjugated paired systems via uniform equivalence $\Phi : X\rightarrow Y$ i.e. $\Phi\circ f = f'\circ \Phi$ and $\Phi\circ g = g'\circ \Phi$.
Suppose that $(f, g)$ has PSP. For a given $\epsilon > 0$, choose $0 < \delta < \epsilon$ such that if $x_{1}, x_{2}\in X$ satisfy $d(x_{1}, x_{2}) < \delta$, then $p(\Phi(x_{1}), \Phi(x_{2})) < \epsilon$. 
Choose $0 < \eta < \delta$ correspond to $\delta$ by PSP of $(f, g)$. 
Choose $0 < \gamma <\eta$ such that if $y_{1}, y_{2}\in Y$ satisfy $p(y_{1}, y_{2}) < \gamma$, then $d(\Phi^{-1}(y_{1}), \Phi^{-1}(y_{2})) < \eta$. 
Let $\lbrace y_{n} \rbrace_{n\in \mathbb{Z}}$ be a $\gamma$-pseudo orbit of $(f', g')$. Therefore $p(f'(y_{n}), g'(y_{n+1})) < \gamma$, for each $n\in \mathbb{Z}$  implying that $ d(\Phi^{-1}f'(y_{n}), \Phi^{-1}g'(y_{n+1})) < \eta$, for each $n\in \mathbb{Z}$ implying that $ d(f\Phi^{-1}(y_{n}), g\Phi^{-1}(y_{n+1})) < \eta$, for each $n\in \mathbb{Z}$. Hence $\alpha = \lbrace \Phi^{-1}(y_{n}) \rbrace_{n\in \mathbb{Z}}$ is an $\eta$-pseudo orbit of $(f, g)$. 
Let $x\in X$ be a $\delta$-tracing point of $\alpha$ through $(f, g)$. 
Therefore $d(f^{n}(x), g^{n}\Phi^{-1}(y_{n})) < \delta$, for each $n\in \mathbb{Z}$ implying that $p(\Phi f^{n}(x), \Phi g^{n}\Phi^{-1}(y_{n}) < \epsilon$, for each $n\in \mathbb{Z}$ implying that $p((f')^{n}\Phi(x), (g')^{n}(y_{n})) < \epsilon$, for each $n\in \mathbb{Z}$. 
Hence $\lbrace y_{n} \rbrace_{n\in \mathbb{Z}}$ can be $\epsilon$-traced by $y = \Phi(x)$ through $(f', g')$.
Therefore $(f', g')$ has PSP. 
\end{proof}

\section{Paired Topological Stability}
In recent days, the appearance of many complex dynamical systems in different domains of Applied Mathematics, Physics, Engineering, Biotechnology etc. unleashed the usefulness of the notion of topological stability in real life problems. Originally, it was introduced by Peter Walters \cite{WA} for a diffeomorphism between compact smooth manifolds. The concept roughly expresses the fact that the dynamics of a homeomorphism remains unaffected by continuous small perturbations. The mathematical definition given by Walters can be formulated for homeomorphisms on compact metric spaces as follows:
A homeomorphism $f$ on a compact metric space $X$ is said to be topologically stable if for each $\epsilon>0$, there exists a $\delta>0$ such that for each homeomorphism $h$ on $X$ satisfying $d_{C^{0}}(f, h) = \sup\limits_{x\in X} d(f(x),h(x))<\delta$, there exists a continuous map $k:X\rightarrow X$ such that $f\circ k = k\circ h$ and $d(k(x),x)<\epsilon$, for each $x\in X$  \cite{WO}.           
\par
\vspace*{0.1cm}

We define a bounded metric $\overline{d}$ on $X$ by $\overline{d}(x, y) = min \lbrace d(x, y), 1\rbrace$, for all $x, y\in X$. For homeomorphisms $f, f'$ on $X$, we denote $\overline{d}(f, f') = \sup\limits_{x\in X}\overline{d}(f(x), f'(x))$. 
For a given $\delta > 0$, we define $UE_{(f,g)}^{\delta}(X)=\lbrace h: X\rightarrow X \mid (h, g)$ is a commutative paired system on $X$ and $\overline{d}(f, gh) < \delta\rbrace$. Clearly every uniform equivalence on $(X, d)$ is a uniform equivalence on $(X, \overline{d})$ and conversely. We now introduce the notion of topological stability for paired systems. 

\begin{Definition}
We say that a paired system $(f,g)$ on $X$ is paired topologically stable (PTS) if for each $\epsilon>0$, there exists a $\delta>0$ such that for each $h\in UE_{(f,g)}^{\delta}(X)$, there exists a continuous map $k:X \rightarrow X$ such that $f \circ k = k\circ g\circ h$ and $d(k(x), x) < \epsilon$, for each $x\in X$. Additionally, if $g\circ k = k\circ g$, then we say that $(f, g)$ is $c$-paired topologically stable.
\label{D2} 
\end{Definition}

Paired topological stability of $(f, g)$ tells us that in the class of all uniform equivalences of $X$ equipped with bounded metric, there is a neighbourhood $U$ of $g^{-1}f$ in which $f$ is fairly simple map from dynamical viewpoint whenever we study it via $g$, i.e. $f$ can be seen via continuous image of $g\circ h$ for every uniform equivalence $h\in U$ commuting with $g$. While dealing with PTS, we can assume that $0<\mathfrak{c}, \mathfrak{c'}, \epsilon, \delta<1$.

\begin{theorem}
For paired dynamical systems, paired topological stability and $c$-paired topological stability  are uniform dynamical properties. 
\label{TE3}
\end{theorem}
\begin{proof}
Let $(X,(f,g))$ and $(Y,(f',g'))$ be uniformly conjugated paired systems via uniform equivalence $\Phi : X\rightarrow Y$ i.e. $\Phi\circ f = f'\circ \Phi$ and $\Phi\circ g = g'\circ \Phi$. 
Let $(f, g)$ be PTS. For a given $\epsilon > 0$, choose $0 < \delta <\epsilon$ such that if $x_{1}, x_{2}\in X$ satisfy $d(x_{1}, x_{2}) < \delta$, then $p(\Phi(x_{1}), \Phi(x_{2}))$ $< \epsilon$. Choose $0 < \eta < \delta$ correspond to $\delta$ by PTS of $(f, g)$. Choose $0< \gamma <\eta$ such that if $y_{1}, y_{2}\in Y$ satisfy $p(y_{1}, y_{2}) < \gamma$, then $d(\Phi^{-1}(y_{1}), \Phi^{-1}(y_{2})) < \eta$. 
Choose $h'\in UE_{(f', g')}^{\gamma}(Y)$. Therefore $d(\Phi^{-1}f'(y), \Phi^{-1}g'h'(y)) = d(f\Phi^{-1}(y), g\Phi^{-1}h'(y))$ $= d(f\Phi^{-1}(y), g\Phi^{-1}h'\Phi(\Phi^{-1}(y))) < \eta$, for each $y\in Y$ implying that $d(f(x), g\Phi^{-1}h'\Phi(x))$ $< \eta$, for each $x\in X$. 
Thus we have $\Phi^{-1} h' \Phi \in UE_{(f, g)}^{\eta}(X)$ and hence by paired topological stability of $(f, g)$, we can choose a continuous map $k: X\rightarrow X$ such that $f\circ k = k\circ \Phi^{-1}\circ h'\circ \Phi\circ g$ and $d(k(x), x) < \delta$, for each $x\in X$. Therefore $d(k\Phi^{-1}(y), \Phi^{-1}(y)) < \delta$, for each $y\in Y$ implying that $d(\Phi k\Phi^{-1}(y), y) < \epsilon$, for each $y\in Y$. Also, $f'\circ \Phi\circ k\circ \Phi^{-1} = \Phi\circ f\circ k\circ \Phi^{-1} = \Phi\circ k\circ \Phi^{-1}\circ h'\circ \Phi\circ g\circ \Phi^{-1} = (\Phi\circ k\circ \Phi^{-1})\circ h'\circ g'$ and hence $(f', g')$ is paired topologically stable. Moreover, if $g\circ k = k\circ g$, then $g'\circ \Phi \circ k\circ \Phi^{-1} = \Phi \circ g \circ k \circ \Phi^{-1} = \Phi \circ k \circ g \circ \Phi^{-1} = \Phi \circ k \circ \Phi^{-1} \circ g'$. Hence if $(f, g)$ is $c$-paired topologically stable, then $(f', g')$ is also $c$-paired topologically stable. 
\end{proof}

Among other results discussed in \cite{WO}, following plays a major role in the theory of topological dynamics. This result studies the connection of expansivity, shadowing and topological stability for homeomorphisms on a compact metric spaces.   

\begin{theorem}[Walters stability theorem] 
Let $f:X\rightarrow X$ be a homeomorphism on a compact metric space $X$. If $f$ is expansive and has shadowing, then $f$ is topologically stable.
\label{T5}    
\end{theorem}      

The validity of above result in different formulation of a dynamical system has been investigated in 
\cite{AKLM, ALLP, ART, CLT, DDS, DLRWS, DKYG, KDA, KDDP1, KLMP, LMT, LNYT, MMTT, TDT}.
We prove the above stability theorem for paired dynamical systems. We first need the following lemma. 

\begin{Lemma}
Let $(f, g)$ be a paired system on a relatively compact metric space $X$ with PSP and $f$ be expansive with expansivity constant $\mathfrak{c}$. Then the following statements are true:
\begin{enumerate}
\item[(i)] For any $x_{0}\in X$ and $\lambda >0$, there exists an $N > 0$  such that if $x\in X$ satisfies $d(x_{0}, x) \geq \lambda$, then $d(f^{n}(x_{0}), f^{n}(x)) > \mathfrak{c}$, for some $|n| \leq N$.
\item[(ii)] Choose $0 < \epsilon< \frac{\mathfrak{c}}{3}$ and $\delta > 0$ correspond to $\epsilon$ by paired shadowing property. Then each $\delta$-pseudo orbit of $(f, g)$ can be $\epsilon$-traced by a unique point of $X$.
\end{enumerate}
\label{L1}
\end{Lemma} 
\begin{proof}
Let $(f, g)$ be a paired system with PSP and $f$ be expansive with expansivity constant $\mathfrak{c}$. 
\begin{enumerate}
\item[(i)] On the contrary, choose a sequence $\lbrace x_{N}\rbrace_{N\in \mathbb{N}}$ in $X$ such that $d(f^{n}(x_{0}), f^{n}(x_{N}))$ $\leq \mathfrak{c}$, for all $|n| \leq N$ and $d(x_{0}, x_{N}) \geq \lambda$. Since $B[x_{0}, \mathfrak{c}]$ is compact, we can assume that $x_{N}$ converges to $x$, for some $x\in X$. But then $d(f^{n}(x_{0}), f^{n}(x))) \leq \mathfrak{c}$, for each $n \in \mathbb{Z}$ and $d(x_{0}, x) \geq \lambda$, a contradiction to the expansivity of $f$.
\item[(ii)] Let $\lbrace x_{n}\rbrace_{n\in\mathbb{Z}}$ be a $\delta$-pseudo orbit of $(f, g)$ and $x,y\in X$ be two points that $\epsilon$-traced this orbit. 
Then $d(f^{n}(x), f^{n}(y)) \leq d(f^{n}(x), g^{n}(x_{n})) + d(g^{n}(x_{n}), f^{n}(y)) < 2\epsilon < \mathfrak{c}$, for each $n\in \mathbb{Z}$. By the expansivity of $f$, we get that $x = y$. 
\end{enumerate}
\end{proof}

\begin{theorem}
Let $(f,g)$ be a paired system on a relatively compact metric space $X$. If $(f,g)$ has PSP and $f$ is expansive with expansivity constant $\mathfrak{c}$, then $(f,g)$ is paired topologically stable. Further, if $(f, g)$ is commutative, then $(f, g)$ is $c$-paired topologically stable. Moreover, for each $0< \epsilon < \frac{\mathfrak{c}}{3}$, there exists a $\delta > 0$ such that for each $h\in UE_{(f,g)}^{\delta}(X)$, there exists a unique continuous map $k : X\rightarrow X$ satisfying $f\circ k = k\circ g\circ h$ and $d(k(x), x)\leq \epsilon$, for each $x\in X$. In addition, if $h$ is expansive with expansivity constant $\mathfrak{c'}\geq 3\epsilon$, then the conjugating map $k$ is injective on the set of all fixed points of $g$ i.e. $k|_{Fix(g)} : Fix(g)\rightarrow X$ defined as $k|_{Fix(g)}(x) = k(x)$, for each $x\in Fix(g) = \lbrace x \mid g(x) = x\rbrace$, is injective. 
\label{T6}
\end{theorem}
\begin{proof}
For $0 < \epsilon < \frac{\mathfrak{c}}{3}$, choose $0 < \alpha < \frac{\epsilon}{2}$ such that if $x, y\in X$ satisfy $d(x, y) < \alpha$, then $d(g(x), g(y)) < \frac{\epsilon}{2}$. Choose $0 < \delta < \alpha$ such that each $\delta$-pseudo orbit of $(f, g)$ can be $\alpha$-traced and hence $\frac{\epsilon}{2}$-traced through $(f, g)$ by a unique point of $X$. Choose $h\in UE_{(f,g)}^{\delta}(X)$. 
Clearly, $\lbrace h^{n}(x)\rbrace_{n\in \mathbb{Z}}$ forms a $\delta$-pseudo orbit of $(f, g)$, for each $x\in X$. 
We use Lemma \ref{L1}(ii) to define $k : X\rightarrow X$ where $k(x)$ is the unique $\alpha$-tracing point of the $\delta$-pseudo orbit $\lbrace h^{n}(x)\rbrace_{n\in \mathbb{Z}}$ i.e. $d(f^{n}(k(x)), g^{n}h^{n}(x)) < \alpha$, for each $n\in \mathbb{Z}$ and for each $x\in X$. 
In particular, if $n = 0$, then $d(k(x), x) < \alpha < \epsilon$, for each $x\in X$. 

Note that $d(f^{n+1}k(x), g^{n+1}h^{n+1}(x))< \alpha$, for each $n\in \mathbb{Z}$ and for each $x\in X$, and $d(f^{n}k(gh(x)), g^{n+1}h^{n+1}(x)) = d(f^{n}(k(gh(x))), g^{n}h^{n}gh(x)) < \alpha$, for each $n\in \mathbb{Z}$ and for each $x\in X$. 
Therefore $d(f^{n}(fk(x)), f^{n}(kgh(x))) < 2\alpha < \epsilon < \mathfrak{c}$, for each $n\in \mathbb{Z}$ and for each $x\in X$. 
Since $f$ is expansive, we get that $f\circ k = k\circ g\circ h$. 
Note that if $(f, g)$ is commutative also, then for each $n\in \mathbb{Z}$, we have
\begin{align*}
d(f^{n}(g\circ k(x)), f^{n}(k\circ g(x))) &= d(gf^{n}(k(x)), f^{n}(kg(x))) \\
&\leq d(gf^{n}(k(x)), gg^{n}h^{n}(x)) + d(gg^{n}h^{n}(x), f^{n}(kg(x))) \\
&= d(gf^{n}(k(x)), gg^{n}h^{n}(x)) + d(g^{n}h^{n}(g(x)), f^{n}(k(g(x)))) \\
&\leq \frac{\epsilon}{2} + \alpha \\
&< \epsilon < \mathfrak{c}
\end{align*}
Again by the expansivity of $f$, we get that $g\circ k = k\circ g$.

We now show that $k$ is continuous at each point $x_{0}\in X$. Choose a $\lambda > 0$. From Lemma \ref{L1}(i), we choose an $N > 0$ such that if $y\in X$ satisfies $d(f^{n}k(x_{0}), f^{n}k(y)) \leq \mathfrak{c}$, for all $|n|\leq N$, then $d(k(x_{0}), k(y))  < \lambda$. 
Choose an $\eta > 0$ such that if $y\in X$ satisfies $d(x_{0}, y) < \eta$, then $d(g^{n}h^{n}(x_{0}), g^{n}h^{n}(y)) < \frac{\mathfrak{c}}{3}$, for all $|n| \leq N$. 
Therefore if $y\in X$ satisfies $d(x_{0}, y) < \eta$, then for all $|n| \leq N$ and for each $y\in X$ we have 
\begin{align*}
d(f^{n}k(x_{0}), f^{n}k(y)) &= d(kg^{n}h^{n}(x_{0}), kg^{n}h^{n}(y)) \\
&\leq d(kg^{n}h^{n}(x_{0}), g^{n}h^{n}(x_{0})) + d(g^{n}h^{n}(x_{0}), g^{n}h^{n}(y)) \\
&\hspace*{0.2cm} + d(g^{n}h^{n}(y), kg^{n}h^{n}(y)) \\
&< \epsilon + \frac{\mathfrak{c}}{3} + \epsilon \\
&< \mathfrak{c}.
\end{align*}

Therefore $d(k(x_{0}), k(y)) < \lambda$ whenever $y\in X$ satisfies $d(x_{0}, y) < \eta$ implying that $k$ is continuous at $x_{0}$. Since $x_{0}$ is chosen arbitrarily, we get that $k$ is a continuous map.

Assume that there exists an another continuous map $k'$ satisfying $f\circ k' = k'\circ g\circ h$ and $d(k'(x), x) < \epsilon$, for each $x\in X$. 
Therefore for each $n\in \mathbb{Z}$ and for each $x\in X$, we have
\begin{align*}
d(f^{n}k(x), f^{n}k'(x)) &\leq d(f^{n}k(x), g^{n}h^{n}(x)) + d(g^{n}h^{n}(x), f^{n}k'(x)) \\
&= d(kg^{n}h^{n}(x), g^{n}h^{n}(x)) + d(g^{n}h^{n}(x), k'g^{n}h^{n}(x)) \\
&< 2\epsilon < \mathfrak{c}.
\end{align*}
Since $f$ is expansive, we get that $k(x) = k'(x)$, for each $x\in X$. Now, assume that $x, y$ are fixed points under $g$, $h$ is expansive, $k(x) =k(y)$ and $n\in \mathbb{Z}$. Note that 
\begin{align*}
d(h^{n}(x), h^{n}(y)) &\leq d(h^{n}(x), k(h^{n}(x))) + d(kh^{n}(x), kg^{n}h^{n}(x)) \\ 
&\hspace*{0.4cm} + d(kg^{n}h^{n}(x), kg^{n}h^{n}(y)) + d(kg^{n}h^{n}(y), kh^{n}(y)) \\
&\hspace*{0.4cm} + d(k(h^{n}(y)), h^{n}(y)) \\
&< \epsilon + 0 + 0 +0 + \epsilon \\
&= 2\epsilon < \mathfrak{c'}.
\end{align*}
Since $h$ is expansive, we get that $x = y$. 
\end{proof}

\begin{Example}
Under the assumptions of Example \ref{E2}, we define $h$ on $X$ by $h(\lbrace x_{i}\rbrace) = \lbrace x'_{i}\rbrace$, where $x'_{i} = (x_{i} + 1)$ $mod$ $2$, for each $i\in \mathbb{Z}$. Since $(f, h)$ is commutative and $h^{2} = Id$, we get that $(f, h)$ has PSP if and only if $h^{-1}f$ has shadowing if and only if $h^{-2}f^{2} = f^{2}$ has shadowing if and only if $f$ has shadowing. Therefore $(f, h)$ is a commutative paired system with paired shadowing property. Since $f$ is expansive, from Theorem \ref{T6} we get that $(f, h)$ is $c$-paired topologically stable.
\label{E4}
\end{Example}

\textbf{Questions :-} The following questions are obvious and still needs to be answered:
\begin{enumerate} 
\item Does there exist a paired system $(f, g)$ such that $(f, g)$ has PSP but $g^{-1}f$ does not have the shadowing?
\item Is the topological stability of $g^{-1}f$ is equivalent to the paired topological stability of $(f,g)$, provided $(f, g)$ is commutative and $g$ is equicontinuous? 
\end{enumerate} 


\end{document}